\documentclass[12pt]{amsart}
\usepackage{latexsym, amsmath, amscd, amssymb, amsthm} 
\usepackage[T2A]{fontenc}
 \usepackage[cp1251]{inputenc}
\usepackage[english]{babel}
\newtheorem{te}{Theorem}[section]

\begin{document}

\noindent

 \title[]{ A note about invariants of algebraic curves }

\author{Leonid Bedratyuk}\address{Khmelnitskiy national university, Insituts'ka, 11,  Khmelnitskiy, 29016, Ukraine}

\begin{abstract} 
For affine algebraic curves we reduce a calculation of its invariants to calculation of the intersection of kernels of some derivations.
\end{abstract}
\maketitle
\section{Introduction}

Consider an affine algebraic  curve  
$$
C: F(x,y)=\sum_{i+j \leq d} a_{i,j} x^i y^j =0, a_{i,j} \in {\bf k},
$$
defined over field ${\bf k},$ $ {\rm char}\, {\bf k} >0.$
Let  ${\bf k}[C]$ and   ${\bf k}(C)$  be the algebras of polynomial and rational function of coefficients of the curve  $C$.
Those  affine  transformations of plane whose  preserve the algebraic form of equation   $F(x,y)$ generate a group $G$  which is subgroup of the group of plane  affine transformations  $A(2).$ A function  $\phi(a_{0,0},a_{1,0},\ldots,a_{d,0}) \in {\bf k}(C) $  is called  $G$-invariant if  
$\phi(\tilde{a}_{0,0},\tilde{a}_{1,0},\ldots,\tilde{a}_{d,0})=\phi(a_{0,0},a_{1,0},\ldots,a_{d,0}),$  where $\tilde{a}_{0,0},$ $\tilde{a}_{1,0},\ldots,$ $\tilde{a}_{d,0}$ are defined   from the condition 
$$F(gx,gy)=\sum_{i+j \leq d} a_{i,j} (gx)^i (gy)^j=\sum_{i+j \leq d} \tilde{a}_{i,j} x^i y^j,$$  for all  $g \in G.$
Curves   $C$ and  $C'$ are  said to be  $G$-isomorphic if they lies on the same   $G$-orbit.

The algebras of all $G$-invariant polynomial and rational functions denotes by  ${\bf k}[C]^G$ and  by  ${\bf k}(C)^G,$  respectively. One  way to find elements of the algebra ${\bf k}[C]^G$ is a specification of invariants of associated  ternary form of order  $d.$ In fact, consider a vector space  $T_d$ generated  by the ternary forms   $\sum\limits_{i+j\leq d}b_{i,j}x^{d-(i+j)} y^i z^{j},$ $b_{i,j} \in {\bf k}$ endowed with the  natural action of the group  $GL_3:=GL_3({\bf k}).$  Given   $f$   $GL_3$-invariant function of  ${\bf k}(T_d)^{GL_3},$ the specification   $f:$ $b_{i,j} \mapsto a_{i,j}$  or $b_{i,j} \mapsto 0$  in the case when  $a_{i,j} \notin {\bf k}(C),$ gives  us  an element of  ${\bf k}(C)^G.$

 But  $SL_3$-invariants(thus and  $GL_3$-invariants) of ternary form are known only for the  cases   $d \leq 4,$ see \cite{Bro}. Furthermore, analyzing of the Poincare series  of the algebra  of invariants of ternary forms, \cite{B_X}, we  see that the algebras are very complicated and  no chance to find theirs mininimal generating sets.

Since ${\bf k}(T_d)^{GL_3}$  coincides with ${\bf k}(T_d)^{\mathfrak{gl}_3}$  it implies that the algebra of invariants is an intersection of kernels of some derivations of the algebra ${\bf k}(T_d).$ Then in  place  of the specification of coefficients of the form we may use a "specification" of those derivations.  

First, consider the motivating example.
Let   
$$
C_3: y^2+a_0x^3+3a_1x^2+3a_2x+a_3=0, 
$$
and  let  $G_0$ be a group generated  by the  translations $x \mapsto \alpha \tilde{x}+b.$
 It is easy to show that   $j$-invariant of the curve  $C_3$  equals (\cite{Sylv86},  p.  46): 
$$
j(C_3)=6912\,{\frac { \left( a_{{0}}a_{{2}}-{a_{{1}}}^{2} \right) ^{3}}{{a_{
{0}}}^{2} \left( 4\,{a_{{1}}}^{3}a_{{3}}-6\,a_{{3}}a_{{0}}a_{{1}}a_{{2
}}-3\,{a_{{1}}}^{2}{a_{{2}}}^{2}+{a_{{3}}}^{2}{a_{{0}}}^{2}+4\,a_{{0}}
{a_{{2}}}^{3} \right) }}.
$$
Up to constant factor the $j(C_3)$  equal to   $\dfrac{S^3}{T}$ where   $S$  and $T$ are the specification of  two   $SL_3$-invariants  of ternary  cubic, see \cite{Stur}, p.173.

From another  side a direct  calculation yields  that the following is true:  $\mathcal{D} \left(j(C_3)\right)=0$  and  ${\mathcal{H}} \left(j(C_3)\right)=0$ where  $\mathcal{D},$ ${\mathcal{H}}$ denote the  derivations of the algebra of rational functions  ${\bf k}(C_3)={\bf k}(a_0,a_1,a_2,a_3):$
 $$\mathcal{D}(a_i)=i a_{i-1},{ \mathcal{H}}(a_i)=(3-i)a_{i}, i=0,1,2,3.$$
 
 From computing  point of view, the calculation of $\ker \mathcal{D}  \cap \ker \mathcal{H} $  is more effective than the calculating of the algebra  of invariants of the 
 ternary cubic.    We will derive further that 
$$ 
\ker \mathcal{D}_3  \cap \ker H_3={\bf k}\left( \frac{\left( a_{{0}}a_{{2}}-{a_{{1}}}^{2} \right) ^{3}}{a_0^3}, \frac{{a_{3}}\,a_0^{2} + 2\,{a_{1}}^{3} - 3\,{a_{1}}\,{a_{2}}\,a_0}{a_0^2}  \right).
$$

In the notes   we prove that for arbitrary algebraic curve  $C$  and its preserved form group  $G$  there exist derivations  $D_i, i\leq 6$ of ${\bf k}(C)$ such that  ${\bf k} (C)^G=\bigcap\limits_i \ker D_i,$  (Theorem 3.2).

In section 1 we give full description of the algebras  of polynomial and rational invariants for the curve $y^2=f(x)$. In section 2  we give exactly form for the  derivations of the action of the Lie algebra $\mathfrak{gl}_3$ on  ${\bf k} (C)$ and give a specification of such action for a curve of the form  $y^2+g(x)y=f(x).$

\section{Invariants of  $y^2=f(x).$}

In some simple cases we may obtain the defining derivation directly.

Consider the curve
$$
C_d: y^2=a_0x^d+da_1 x^{d-1}+\cdots +a_d =\sum_{i=0}^d a_d {d \choose i} x^{d-i},
$$
and let $G$  be the group generated by  the following  transformations 
$$
x=\alpha \tilde{x}+b, y=\tilde{y}.
$$
The algebra ${\bf k}(C_d)^G$ consists of functions  $\phi(a_0,a_1,\ldots,a_d)$ that have the invariance property
$$
\phi(\tilde{a}_0,\tilde{a}_1,\ldots,\tilde{a}_d)=\phi(a_0,a_1,\ldots,a_d).
$$
Here $\tilde{a}_i$ denote  the coefficients of the curve  $\tilde{C}_d:$
$$
\tilde{C}_d:\sum_{i=0}^d a_d {d \choose i} (\alpha \tilde{x}+b)^{d-i}=\sum_{i=0}^d \tilde{a}_d {d \choose i} \tilde{x}^{d-i}.
$$
The coefficients $\tilde{a}_i$   are given by the formulas
\begin{gather}\label{1}
\tilde{a}_i=\alpha^{n-i}\sum_{k=0}^k {i \choose k} b^k a_{n{-}k}.
\end{gather}
The following statement holds
\begin{te}  We   have 
$$
{\bf k}(C)^G=\ker \mathcal{D}_d \cap \ker \mathcal{E}_d,
$$
where $\mathcal{D}_d,$ $\mathcal{E}_d$  denote the following derivations of the algebra ${\bf k}(C):$
$$
\mathcal{D}_d(a_i)=i a_{i-1}, \mathcal{E}_d(a_i)=(d-i) a_i.
$$
\end{te}

Recall that a linear map  $D:{\bf k}(C) \to {\bf k}(C)$  is called a derivation of the algebra  ${\bf k}(C)$  if  $D(f g)=D(f)g+f D(g),$  for all $f,g \in {\bf k}(C).$ The subalgebra $\ker D := \{ f \in {\bf k}(C) \mid D(f)=0 \}$ is called the kernel of the derivation $D.$ The above derivation $\mathcal{D}_d$ is called the basic Weitzenb\"ock derivation.
\begin{proof}
Acting in classical manner, we differentiate with respect to   $b$ both sides of  the equality
$$\phi(\tilde{a}_0,\tilde{a}_1,\ldots,\tilde{a}_d)= \phi(\tilde{a}_0,\tilde{a}_1,\ldots,\tilde{a}_d),$$
and  obtain in this way  
$$
\frac{\partial \phi(\tilde{a}_0,\tilde{a}_1,\ldots,\tilde{a}_d)}{\partial \tilde{a}_0} \dfrac{\partial \tilde{a}_0}{\partial b}+\frac{\partial \phi(\tilde{a}_0,\tilde{a}_1,\ldots,\tilde{a}_d)}{\partial \tilde{a}_1}\dfrac{\partial \tilde{a}_1}{\partial b}+\cdots+\frac{\partial \phi(\tilde{a}_0,\tilde{a}_1,\ldots,\tilde{a}_d)}{\partial \tilde{a}_d}\dfrac{\partial \tilde{a}_d}{\partial b}=0.
$$
Substitute   $\alpha=1,$ $b=0$ to  $\phi(\tilde{a}_0,\tilde{a}_1,\ldots,\tilde{a}_d)$ and taking into account that    $\dfrac{\partial \tilde{a}_i}{\partial b}\Bigl |_{b=0}=i a_{i-1},$  we get:
$$
\tilde{a}_0 \frac{\partial \phi(\tilde{a}_0,\tilde{a}_1,\ldots,\tilde{a}_d)}{\partial \tilde{a}_1} +2\tilde{a}_1 \frac{\partial \phi(\tilde{a}_0,\ldots,\tilde{a}_d)}{\partial \tilde{a}_2}+\cdots d \tilde{a}_{d-1} \frac{\partial \phi(\tilde{a}_0,\ldots,\tilde{a}_d)}{\partial \tilde{a}_d} =0
$$
Since the function  $\phi(\tilde{a}_0,\ldots,\tilde{a}_d)$ depends  on the variables  $\tilde{a}_i$  in the exact same way as the function  $\phi(a_0,a_1,\ldots,a_d)$ depends on the  $a_i$ then it implies  that  $\phi(a_0,a_1,\ldots,a_d)$  satisfies the differential equation
\begin{gather*}\label{D1}
a_0 \frac{\partial \phi({a}_0,a_1\ldots,{a}_d)}{\partial {a}_1} +2a_1 \frac{\partial \phi({a}_0,a_1\ldots,{a}_d)}{\partial {a}_2}+d a_{d-1} \frac{\partial \phi({a}_0,a_1\ldots,{a}_d)}{\partial {a}_d} =0
\end{gather*}
Thus, $\mathcal{D}_d(\phi)=0.$
Now we differentiate with respect to    $\alpha$ both sides of  the same equality $$\phi(\tilde{a}_0,\tilde{a}_1,\ldots,\tilde{a}_d)= \phi(\tilde{a}_0,\tilde{a}_1,\ldots,\tilde{a}_d).$$
$$
\frac{\partial \phi(\tilde{a}_0,\tilde{a}_1,\ldots,\tilde{a}_d)}{\partial \tilde{a}_0} \dfrac{\partial \tilde{a}_0}{\partial \alpha}+\frac{\partial \phi(\tilde{a}_0,\tilde{a}_1,\ldots,\tilde{a}_d)}{\partial \tilde{a}_1}\dfrac{\partial \tilde{a}_1}{\partial \alpha}+\cdots+\frac{\partial \phi(\tilde{a}_0,\tilde{a}_1,\ldots,\tilde{a}_d)}{\partial \tilde{a}_d}\dfrac{\partial \tilde{a}_d}{\partial \alpha}=0.
$$
Substitute   $\alpha=1,b=0,$ to  $\phi(\tilde{a}_0,\tilde{a}_1,\ldots,\tilde{a}_d)$ and taking into account  $\dfrac{\partial \tilde{a}_i}{\partial \alpha}\Bigl |_{ \alpha=1 \,
b=0}=(d-i) a_{i},$  we get:
$$
\tilde{a}_0 \frac{\partial \phi(\tilde{a}_0,\tilde{a}_1,\ldots,\tilde{a}_d)}{\partial \tilde{a}_0} +(d-1)\tilde{a}_1 \frac{\partial \phi(\tilde{a}_0,\ldots,\tilde{a}_d)}{\partial \tilde{a}_1}+\cdots + \tilde{a}_{d-1} \frac{\partial \phi(\tilde{a}_0,\ldots,\tilde{a}_d)}{\partial \tilde{a}_{d-1}} =0
$$
It implies that  
 $\mathcal{E}_d(\phi({a}_0,a_1\ldots,{a}_d))=0.$
\end{proof}

The derivation  $\mathcal{E}_d$ sends the monom $a_0^{m_0} a_1^{m_1}\cdots a_d^{m_d}$ to the  term   $$(m_0 d+m_1 (d-1)+\cdots m_{d-1}) a_0^{m_0} a_1^{m_1}\cdots a_d^{m_d}.$$  Let the number  $\omega \left( a_0^{m_0} a_1^{m_1}\cdots a_d^{m_d} \right):= m_0 d+m_1 (d-1)+\cdots m_{d-1}$ be called the weight of the monom    $a_0^{m_0} a_1^{m_1}\cdots a_d^{m_d}.$  In particular $\omega(a_i)=d-i.$ 

 A homogeneous polynomial  $f \in {\bf k}[C]$ be called isobaric if all their monomial have equal weights. A weight $\omega(f)$  of an isobaric polynomial  $f$ is called a weight of its monomials.  Since   $\omega(f)>0,$  then  ${\bf k}[C]^{\mathcal{E}_d}=0.$ It implies  that ${\bf k}[C]^{G}=0.$ 
 
 If   $f,g$ are two isobaric polynomials  then 
$$
\mathcal{E}_d\left(\frac{f}{g}\right)=(\omega(f)-\omega(g))\frac{f}{g}.
$$
Therefore the algebra  $k(C)^{\mathcal{E}_d}$ is generated  by rational fractions which both denominator and numerator has equal weight.

The kernel of the derivation $\mathcal{D}_d$  also is well-known, see  \cite{Now}, \cite{Aut}, and 
$$
\ker \mathcal{D}_d={\bf k}(a_0,z_2,\ldots,z_d),
$$
where 
$$
z_i:= \sum_{k=0}^{i-2} (-1)^k {i \choose k} a_{i-k}  a_1^k a_0^{i-k-1} +(i-1)(-1)^{i+1} a_1^i, i=2,\ldots,d.
$$
In particular, for $d=5$, we get
$$
\begin{array}{l}
z_2={a_{2}}\,a_0 - {a_{1}}^{2}
\\
z_3={a_{3}}\,a_0^{2} + 2\,{a_{1}}^{3} - 3\,{a_{1}}\,{a_{2}}\,a_0
\\
z_4={a_{4}}\,a_0^{3} - 3\,{a_{1}}^{4} + 6\,{a_{1}}^{2}\,{a_{2}}\,a_0 - 4
\,{a_{1}}\,{a_{3}}\,a_0^{2}
\\
z_5={a_{5}}\,a_0^{4} + 4\,{a_{1}}^{5} - 10\,{a_{1}}^{3}\,{a_{2}}\,a_0 + 
10\,{a_{1}}^{2}\,{a_{3}}\,a_0^{2} - 5\,{a_{1}}\,{a_{4}}\,a_0^{3}
.
\end{array}
$$
It is easy to see that  $\omega(z_i)=i (n-1).$ The following element $\dfrac{z_i^d}{a_0^{i(d-1)}}$ has the zero weight for any $i.$ Therefore the statement  holds:
\begin{te}
$$
{\bf k}(C_d)^G={\bf k}\left( \frac{z_2^d}{a_0^{2(d-1)}},\frac{z_3^d}{a_0^{3(d-1)}}, \cdots, \frac{z_d^d}{a_0^{d(d-1)}}\right).
$$
\end{te}
For the curve 
$$
C_d^0: y^2=x^d+da_1 x^{d-1}+\cdots +a_d =x^d+\sum_{i=1}^d a_d {d \choose i} x^{d-i}.
$$
and for  the group  $G_0$ generated by translations  $x=\tilde{x}+b,$ the algebra of invariants becomes simpler:
$${\bf k}\left({C^0}_d\right)^{G_0}={\bf k}(z_2,z_3,\ldots,z_d).$$
 
\begin{te}
$(i)$ For arbitrary set of   $d-1$  numbers  $j_2,$$j_3,\ldots,j_d$ there exists a curve  $C$  such that  
$z_i(C)=j_i.$ 

$(ii)$ For two curves   $C$ and  $C'$   the equalities  $z_i(C)=z_i(C')$ hold for  $ 2 \leq i \leq d,$  if and only if these curves are $G_0$-isomorphic.
\end{te}
\begin{proof}
$(i).$  Consider the system of equations 
$$
\left\{
\begin{array}{l}
{a_{2}} - {a_{1}}^{2}=j_2
\\
{a_{3}} + 2\,{a_{1}}^{3} - 3\,{a_{1}}\,{a_{2}}=j_3
\\
{a_{4}} - 3\,{a_{1}}^{4} + 6\,{a_{1}}^{2}\,{a_{2}} - 4
\,{a_{1}}\,{a_{3}}=j_4
\\
\ldots \\
\displaystyle a_d+\sum_{k=1}^{d-2} (-1)^k {d \choose k} a_{d-k}  a_1^k  +(d-1)(-1)^{d+1} a_1^d=j_d
\end{array}
\right.
$$
By solving  it  we obtain  
\begin{gather}\label{2}
a_n=j_n+\sum_{i=2}^n {n \choose i} a_1^k j_{n-k}
\end{gather}
Put  $a_1=0$  we get  $a_n=j_n$,\, i.e. the curve 
$$
C: y^2=x^d+{d \choose 2} j_2 x^{d-2}+\cdots +j_d,
$$
is desired  one.

$(ii).$ We may assume, without loss of generality,    that the curve  $C$  has the form
$$
C: y^2=x^d+{d \choose 2} j_2 x^{d-2}+\cdots +j_d.
$$
Suppose that for a curve  
$$
C': y^2=x^d+da_1 x^{d-1}+\cdots +a_d =x^d+\sum_{i=1}^d a_d {d \choose i} x^{d-i}.
$$
holds $z_i(C')=z_i(C)=j_i.$
Comparing  (\ref{2})  with  (\ref{1})  we deduce that the curve   $C'$ is obtained from the curve  $C$  by the  translation  $x+a_1.$ 
\end{proof}

\section{General case  and invariants of  $y^2+g(x)y=f(x)$}

Consider the  vector ${\bf k}$-space  $T_d$ of ternary form of degree  $d:$
$$ 
u(x,y,z)=\sum_{i+j\leq d} \, \frac{d!}{i! j! (d{-}(i+j))!}a_{i, j}\, x^{d-(i+j)} y^i z^j,
$$
where  $a_{i, j} \in {\bf k}.$
Let us identify in the natural way the algebra of rational function  ${\bf k}(T_d)$  on the vector space $T_d$ with the algebra  of polynomials of the $\displaystyle \frac{1}{2} (d+1)(d+2)$ variables. The natural action of the group  $GL_3$ on   $T_d$  induced the action of  $GL_3$ (and the Lie algebra $\mathfrak{gl_{3}}$) on   ${\bf k}[T_d]$.  
The corresponding algebra of invariants   ${\bf k}(T_d)^{GL_3}={\bf k}(T_d)^{\mathfrak{gl_{3}}}$  is called the algebra of   $GL_3$-invariants (or absolute invariants) of ternary form of degree  $d.$  The following statement  holds:
\begin{te}
$$ {\bf k}(T_d)^{GL_3}=\ker D_1 \cap \ker D_2  \cap \ker \hat{D}_1 \cap \ker \hat{D}_2 \cap  \ker E_1 \cap \ker E_2 \cap \ker E_3$$
where
$$
\begin{array}{ll}
\displaystyle \! \! D_1(a_{i,j})=i\,a_{i{-}1,j}, & \! \! D_2(a_{i,j})=j\,a_{i{+}1,j{-}1},\\
\displaystyle \! \! \hat D_1(a_{i,j})=(n-(i+j))\,a_{i{+}1,j}, & \! \!  \hat D_2(a_{i,j})=i\,a_{i{-}1,j{+}1},\\
   \hat{D}_3(a_{i,j})=(n-(i+j)) a_{i,j+1}, & \! \! D_3(a_{i,j})=j a_{i,j-1}, \\
\! \! E_1(a_{i,j})=(n-(2i+j)) a_{i,j}, & \! \! E_2(a_{i,j})=i a_{i,j},
\end{array}
$$
$$E_3(a_{i,j})=j a_{i,j}. $$
\end{te}
\begin{proof}
The Lie algebra  $\mathfrak{gl}_3$ acts on the vector space of ternary form  $T_d$ by derivations,namely 
$$
\begin{array}{ll}
\displaystyle D_1=-y \frac{\partial}{\partial x}, &  \displaystyle D_2=-z \frac{\partial}{\partial y}, \\
E_1=-x\frac{\partial}{\partial x}, &E_2= -y\frac{\partial}{\partial y},   \\
\displaystyle \hat D_1=-x \frac{\partial}{\partial y}, &  \displaystyle \hat D_2=-y \frac{\partial}{\partial z}, \\
 \displaystyle D_3=-z \frac{\partial}{\partial x}, & E_3= -\displaystyle  z \frac{\partial}{\partial z},
\end{array}
$$
$$
\begin{array}{lll}
 \displaystyle &  \displaystyle \hat D_3=-x \frac{\partial}{\partial z}.  &
\end{array}
$$
To extend  the actions  $\mathfrak{gl}_3$  to  the algebra  ${\bf k}(T_d)$ we use the well-known   fact  of classical invariant theory  that the generic form  $u(x,y,z)$ is a covariant. It means that any of above derivation (considered as derivation of ${\bf k}[T_d,x,y,z]$) must kill the form. 
In particular,  for the derivation  $D_1$  we have
\begin{gather*}
D_1(u(x,y,z))=
\sum_{i+j\leq n} \, \frac{n!}{i! j! (n{-}(i+j))!}(D_1(a_{i,j}) x^{n-(i+j)} y^i z^j+
+a_{i,j} D_1(x^{n-(i+j)} y^i z^j))=\\
=D_1(a_{0,1}) x^{n-1} x_3+\cdots+D_1(a_{0,n}) \frac{1}{n!} z^n+ 
+\sum_{\begin{array}{c} \mbox{\small \it i+j}\leq n \\  i>0 \end{array} } \Bigl( D_1(a_{i,j})-i\,a_{i{-}1,j}\Bigr) x^{d-(i+j)} y^i z^j.
\end{gather*}
It following that the equality  $D_1(u(x,y,z))=0$  is possible only if all coefficients are equal to zero. Therefore we get  $D_1(a_{0,j})=0$ for all  $0\leq j \leq n,$ and    $D_1(a_{i,j})=i\,a_{i{-}1,j} $ as required.  In  an exactly similar way  we will obtain actions on ${\bf k}(T_d)$ for  the rest derivations. 
\end{proof}

\noindent
{\bf Corollary.}  ${\rm tr \,deg}_k\,\,{\bf k}(T_{d})^{GL_3} \leq \displaystyle \frac{1}{2} (d+1)(d+2)-7.$
\begin{proof}
It follows from the fact  that for a nonzero derivation $D$ of  polynomial  ring $R$  of $n$  variables ${\rm tr \,deg} \ker D \leq n-1,$  see \cite{Now}, Proposition 7.1.1.
\end{proof}

An obvious consequence of the theorem is the following:  

\begin{te} For  affine algebraic curves   $d$

$$
C: \sum_{i+j=d} \frac{d!}{i! j! (d-(i+j))!} a_{i,j} x^{d-(i+j)}y^i=0, a_{d,0}=0, a_{i,j} \in {\bf k}, {\rm char} {\bf k} >0,
$$
the following holds:

(i) if $a_{d,0} \neq 0$  and $a_{0,0} \neq 0$   then
$$
{\bf k}(C)^G=\ker D_1 \cap \ker D_2  \cap \ker \hat{D}_1 \cap  \ker E_1 \cap \ker E_2 ,  
$$

(ii) if $a_{d,0} =0 0$ $a_{0,0} \neq 0$  then

$$
{\bf k}(C)^G=\ker D_2 \cap \ker D_3 \cap \ker \hat{D}_1 \cap  \ker E_1 \cap \ker E_2 ,  
$$
where
\begin{align*}
D_1(a_{i,j})&=i a_{i-1,j},D_2(a_{i,j})=j a_{i+1,j-1}, \hat{D}_1(a_{i,j})=(d-(i+j))a_{i+1,j},\\
&\\
E_1(a_{i,j})&=(d-(i+j)) a_{i,j}, E_2(a_{i,j})=i a_{i,j},D_3(a_{i,j})=j a_{i,j-1}.
\end{align*}
\end{te}
\begin{proof}
$(i)$   Consider the associate projective plane curve in $\mathbb{P}^2:$
$$
\sum_{i+j=d} \frac{d!}{i! j! (d-(i+j))!} a_{i,j} X^{d-(i+j)}Y^iZ^j=0, a_{d,0} \neq 0,  a_{0,0} \neq 0. 
$$
The  transformations $X \mapsto \alpha X+ \beta Y+ b Z,$ $Y \mapsto \gamma X+ \delta Y+a Z,$  $Z \mapsto Z$ generate of a subgroup of $GL_3$ which preserve the algebraic form of the equation  of the curve. Therefore the algebra of invariants of the curve (and corresponding affine curve) coincides with the intersection of the kernels of the five derivations $ D_1, D_2, \hat{D}_1, E_1, E_2,$  ($[D_1,D_2]=D_3$).

$(ii).$ For  this case  the   transformations are as follows: $X \mapsto \alpha X+  b Z,$ $Y \mapsto \gamma X+ \delta Y+a Z,$  $Z \mapsto Z $ and  we  have to  exclude  the  derivation $D_1.$
\end{proof}

For  the curve 
\begin{gather*}
\mathcal{C'}_{d}:\frac{d(d-1)}{2} a_{2,d-2} y^2+\sum_{i=0}^{d-1}\frac{ d!}{i!(d-(1+i))!}a_{1,i}x^{d-(i+1)}y+\sum _{i=0}^{d}{\frac {d!\,}{i!\, \left( d-i
 \right) !}a_{{0,i}}{x}^{d-i}}=0,
\end{gather*}
and  for the group $G$ generated by $x \mapsto \alpha x+a, y \mapsto \beta y+b$ we have 
$$
{\bf k}(\mathcal{C'}_{d})^G=\ker D_2 \cap \ker D_3 \cap \ker \hat{D}_1 \cap  \ker E_1 \cap \ker E_2 ,  
$$
and ${\rm tr \,deg}_k\,\,{\bf k}(\mathcal{C'}_{d})^G \leq  2d-3.$

{\bf Example.} Let us calculate the invariants of curve  $C'_5$
\begin{gather*}
10\,{y}^{2}+(5\,a_{{1,0}}{x}^{4}+20\,a_{{1,1}}{x}^{3}+30\,a_{{1,2}}{x
}^{2}+20\,a_{{1,3}}x+5\,a_{{1,4}})y=\\
={x}^{5}+5\,a_{{0,1}}{x}^{4}+10\,a_{{0,2}}{x}^{3}+10\,a_{{0,3}}{x}^{2}+5\,a_{{0,4}}x+a_{{0,5}},
\end{gather*}
with respect to the group  $G_0$ generated by the translations $x=\tilde{x}+a, y=\tilde{y}+b.$  Theorem 3.2 implies  that  $C_5^G=\ker D_3 \cap \ker D_2,$ where the derivations  $D_2, D_3$ act by
\begin{align*}
 &D_{{2}} \left( a_{{1,1}} \right) =0,D_{{2}} \left( a_{{1,0}} \right) =0,D_{{2}} \left( a_{{0,1}} \right) =-a_{{1,0}},D_{{2}}
 \left( a_{{1,2}} \right) =0,D_{{2}} \left( a_{{0,2}} \right) =-2\,a_{
{1,1}},\\&D_{{2}} \left( a_{{0,3}} \right) =-3\,a_{{1,2}},D_{{2}} \left( 
a_{{0,5}} \right) =-5\,a_{{1,4}},D_{{1}} \left( a_{{0,4}} \right) =-4
\,a_{{1,3}},D_{{2}} \left( a_{{1,3}} \right) =0,D_{{2}} \left( a_{{1,4
}} \right) =4.
\end{align*}
and 
\begin{align*}
&D_{{3}} \left( a_{{1,2}} \right) =2\,a_{{1,1}},D_{{3}} \left( a_{{1,1}} \right) =a_{{1,0}},D_{{2}} \left( a_{{1,0}} \right) 
=0,D_{{2}} \left( a_{{0,2}} \right) =2\,a_{{0,1}},D_{{3}} \left( a_{{0
,1}} \right) =1,\\
&D_{{2}} \left( a_{{0,3}} \right) =3\,a_{{0,2}},D_{{3}}
 \left( a_{{0,5}} \right) =5\,a_{{0,4}},D_{{2}} \left( a_{{0,4}}
 \right) =4\,a_{{0,3}},D_{{2}} \left( a_{{1,3}} \right) =3\,a_{{1,2}},
D_{{3}} \left( a_{{1,4}} \right) =4\,a_{{1,3}}.
\end{align*}
By using the Maple command {\tt pdsolve()}  we obtain that 
$$
{\bf k}(C'_5)^{G_0}={\bf k}(g_1,g_2,g_3,g_4,g_5,g_6,g_7),{\bf k}[C'_5]^{G_0}={\bf k}[g_1,g_2,g_3,g_4,g_5,g_6,g_7],
$$
where
\begin{gather*}
g_1=a_{1,0},\\
g_2={a_{{1,0}}}^{2}a_{{0,2}}+{a_{{1,1}}}^{2}-2\,a_{{1,1}}a_{{1,0}}a_{{0,1}},\\
g_3=a_{{1,2}}-2\,a_{{1,1}}a_{{0,1}}+a_{{1,0}}a_{{0,2}},\\
g_4=6\,{a_{{1,1}}}^{2}a_{{0,1}}a_{{1,0}}-4\,{a_{{1,1}}}^{3}-3\,{a_{{1,0}}}
^{2}a_{{0,2}}a_{{1,1}}-3\,a_{{1,2}}{a_{{1,0}}}^{2}a_{{0,1}}+3\,a_{{1,0
}}a_{{1,1}}a_{{1,2}}+a_{{0,3}}{a_{{1,0}}}^{3},\\
g_5=2\,{a_{{1,1}}}^{3}-3\,a_{{1,0}}a_{{1,1}}a_{{1,2}}+a_{{1,3}}{a_{{1,0}}}
^{2},\\
g_6=3\,{a_{{1,0}}}^{4}{a_{{0,2}}}^{2}+{a_{{1,0}}}^{4}a_{{0,4}}-12\,{a_{{1,0
}}}^{3}a_{{1,1}}a_{{0,1}}a_{{0,2}}-4\,{a_{{1,0}}}^{3}a_{{1,1}}a_{{0,3}
}-4\,{a_{{1,0}}}^{3}a_{{1,3}}a_{{0,1}}-\\-12\,{a_{{1,0}}}^{3}a_{{0,2}}a_{
{1,2}}+12\,{a_{{1,1}}}^{2}{a_{{1,0}}}^{2}{a_{{0,1}}}^{2}+24\,{a_{{1,1}
}}^{2}{a_{{1,0}}}^{2}a_{{0,2}}+4\,a_{{1,1}}{a_{{1,0}}}^{2}a_{{1,3}}+36
\,a_{{1,1}}{a_{{1,0}}}^{2}a_{{0,1}}a_{{1,2}}-\\-24\,a_{{1,0}}a_{{1,2}}{a_
{{1,1}}}^{2}-48\,a_{{1,0}}{a_{{1,1}}}^{3}a_{{0,1}}+24\,{a_{{1,1}}}^{4},\\
g_7={a_{{1,0}}}^{4}{a_{{0,2}}}^{2}+{a_{{1,0}}}^{3}a_{{1,4}}-4\,{a_{{1,0}}}
^{3}a_{{1,1}}a_{{0,1}}a_{{0,2}}+6\,{a_{{1,0}}}^{3}a_{{0,2}}a_{{1,2}}+4
\,{a_{{1,1}}}^{2}{a_{{1,0}}}^{2}{a_{{0,1}}}^{2}-\\-4\,{a_{{1,1}}}^{2}{a_{
{1,0}}}^{2}a_{{0,2}}-12\,a_{{1,1}}{a_{{1,0}}}^{2}a_{{0,1}}a_{{1,2}}-4
\,a_{{1,1}}{a_{{1,0}}}^{2}a_{{1,3}}+4\,{a_{{1,0}}}^{2}a_{{0,1}}-4\,a_{
{1,1}}a_{{1,0}}+\\+12\,a_{{1,0}}a_{{1,2}}{a_{{1,1}}}^{2}+8\,a_{{1,0}}{a_{
{1,1}}}^{3}a_{{0,1}}-8\,{a_{{1,1}}}^{4}
\end{gather*}


\end{document}